\documentclass[a4paper,11pt]{amsart}
\usepackage{amstext,amsfonts,amsthm,graphicx,amssymb,amscd,epsfig}
\usepackage{amsmath}
\usepackage[ansinew]{inputenc}    
\usepackage{psfrag}
\usepackage{mathrsfs}
\usepackage{xypic}
\usepackage{graphicx}
\usepackage[all,knot,arc,import,poly]{xy}

\theoremstyle{definition}
\newtheorem{definition}{Definition}[section]
\newtheorem{ex}[definition]{Example}

\theoremstyle{plain}
\newtheorem{prop}[definition]{Proposition}

\newtheorem{coro}[definition]{Corollary}
\newtheorem{teo}[definition]{Theorem}

\newfont{\bbb}{msbm10 scaled\magstephalf}     

\title{The flat geometry of the $I_{1}$ singularity: $(x,y)\mapsto(x,xy,y^{2},y^{3})$}

\author{P. Benedini Riul, R. Oset Sinha}

\date{}

\address{Instituto de Ci\^encias Matem\'aticas e de Computa\c{c}\~ao - USP,
Av. Trabalhador s\~ao-carlense, 400 - Centro,
CEP: 13566-590 - S\~ao Carlos - SP, Brazil}

\email{benedini@usp.br}

\address{Departament de Matem\`{a}tiques,
Universitat de Val\`encia, Campus de Burjassot, 46100 Burjassot,
Spain}

\email{raul.oset@uv.es}

\thanks{Work of P. Benedini Riul supported by CAPES - PVE  88887.122685/2016-00}
\thanks{Work of R. Oset Sinha partially supported by DGICYT Grant MTM2015--64013--P}


\begin{document}

\begin{abstract}
We study the flat geometry of the least degenerate singularity of a singular surface in $\mathbb R^4$, the $I_{1}$ singularity parametrised by $(x,y)\mapsto(x,xy,y^{2},y^{3})$. This singularity appears generically when projecting a regular surface in $\mathbb R^5$ orthogonally to $\mathbb R^4$ along a tangent direction. We obtain a generic normal form for $I_1$ invariant under diffeomorphisms in the source and isometries in the target. We then consider the contact with hyperplanes by classifying submersions which preserve the image of $I_1$. The main tool is the study of the singularities of the height function.
\end{abstract}

\maketitle

\section{Introduction}

Singularity theory has played an important role on recent results on the differential geometry of singular surfaces.
The geometry of the cross-cap (or Whitney umbrella), for
instance, has been studied in depth:
\cite{BruceWest,DiasTari,FukuiHasegawa,HHNUY,HHNSUY,BallesterosTari,OsetSinhaTari}.
Also, the cuspidal edge, the most simple type of wave front, appears
in many papers:
\cite{KRSUY,MartinsSaji,MartinsSaji2,NUY,OsetSinhaTari2,SUY}.

In \cite{MartinsBallesteros} the authors investigate the second order geometry of corank $1$ surfaces in $\mathbb{R}^{3}$. Also, singular surfaces in $\mathbb{R}^{4}$ have been taken into account in \cite{Benedini/Sinha/Ruas}, where corank $1$ surfaces are the main object of study. In that paper, the curvature parabola is defined, inspired by the curvature parabola for corank $1$ surfaces in $\mathbb{R}^{3}$ (\cite{MartinsBallesteros}) and the curvature ellipse for regular surfaces in $\mathbb{R}^{4}$ (\cite{Little}). This curve is a plane curve that may degenerate into a
half-line, a line or even a point and whose trace lies in the normal
hyperplane of the surface. This special curve carries all the second
order information of the surface at the singular point. Singular surfaces in $\mathbb{R}^{4}$ appear naturally as
projections of regular surfaces in $\mathbb{R}^{5}$ along tangent
directions. In this context, the authors associate to a regular surface $N\subset\mathbb{R}^{5}$ a corank $1$ surface $M\subset\mathbb{R}^{4}$ and a regular surface $S\subset\mathbb{R}^{4}$. Furthermore, they compare the geometry of both surfaces $M$ and $S$. An invariant called umbilic
curvature (invariant under the action of
$\mathcal{R}^{2}\times\mathcal{O}(4)$, the subgroup of $2$-jets of
diffeomorphisms in the source and linear isometries in the target)
is defined as well and used to study the singularities of the height function of corank $1$ surfaces in $\mathbb{R}^{4}$.

In \cite{Rieger}, the authors give a classification of all $\mathcal{A}$-simple map germs
$f:(\mathbb{R}^{2},0)\rightarrow(\mathbb{R}^{4},0)$. The singularity $I_{k}$ given by $(x,y)\mapsto(x,xy,y^{2},y^{2k+1})$, $k\geqslant1$ is the first singular germ to appear in this classification. In \cite{Benedini/Sinha/Ruas}, it is shown that this singularity is the only one whose curvature parabola is a non degenerate parabola. Also, when we consider $k=1$, the singularity $I_{1}$ has an interesting geometric property. In \cite{Fuster/Ruas/Tari}, the authors show that given a regular surface $N\subset\mathbb{R}^{5}$, a tangent direction $\textbf{u}$, in a point whose second fundamental form has maximal rank, is asymptotic if and only if the projection of $N$ along $\textbf{u}$ to a transverse $4$-space has a $\mathcal{A}$-singularity worse than $I_{1}$. In a way, $I_{1}$ is to singular surfaces in $\mathbb{R}^{4}$ what the cross-cap is to singular surfaces in $\mathbb{R}^{3}$.

In this paper, we investigate the flat geometry of the singularity $I_{1}$, using its height function and providing geometric conditions for each possible singularity.  Sections 2 and 3 are an overview of the
differential geometry of regular surfaces in $\mathbb{R}^{4}$ and of the the geometry of corank $1$ surfaces in $\mathbb{R}^{4}$, respectively. We bring all the definitions and results from \cite{Benedini/Sinha/Ruas} that are going to be used throughout the paper.

The last section presents our results regarding the flat geometry of a surface whose local parametrisation is $\mathcal{A}$-equivalent to the singularity $I_{1}$. We classify submersions $(\mathbb{R}^{4},0)\rightarrow(\mathbb{R},0)$ up to changes of coordinates in the source that preserve the model surface $\texttt{X}$ parametrised by $I_{1}$ (Theorem \ref{classification}). Such changes of coordinates form a geometric
subgroup $\mathcal{R}(\texttt{X})$ of the Mather group $\mathcal{R}$ (see \cite{Bruce/Roberts,Damon}). Moreover, we study the singularities of the height function of a singular surface whose parametrisation is given by a generic normal form obtained by changes of coordinates in the source and isometries in the target (Theorem \ref{teo.normal-form}). These singularities are modeled by the ones of the submersions obtained before. Finally, we provide geometrical characterizations for each type of singularity of the height function.

Aknowledgements: the authors would like to thank Professor Maria Aparecida Soares Ruas for her suggestions.


\section{The geometry of regular surfaces in $\mathbb{R}^{4}$}\label{regsurfaces}

In this section we present some aspects of regular surfaces in $\mathbb{R}^{4}$. For more details, see \cite{Livro}.
Little, in \cite{Little}, studied the second order geometry of submanifolds immersed in Euclidean spaces, in particular of immersed surfaces in $\mathbb{R}^{4}$. This paper has inspired
a lot of research on the subject (see \cite{BruceNogueira,BruceTari,GarciaMochidaFusterRuas,MochidaFusterRuas,MochidaFusterRuas2,BallesterosTari,OsetSinhaTari,RomeroFuster}, amongst others). Given a smooth surface $S\subset\mathbb{R}^{4}$ and $f:U\rightarrow\mathbb{R}^{4}$
a local parametrisation of $S$ with $U\subset\mathbb{R}^{2}$ an open subset, let
$\{\textbf{e}_{1},\textbf{e}_{2},\textbf{e}_{3},\textbf{e}_{4}\}$ be an orthonormal frame of $\mathbb{R}^{4}$ such that at any $u\in U$,
$\{\textbf{e}_{1}(u),\textbf{e}_{2}(u)\}$ is a basis for $T_{p}S$ and $\{\textbf{e}_{3}(u),\textbf{e}_{4}(u)\}$ is a basis for
$N_{p}S$ at $p=f(u)$.
The second fundamental form of $S$ at $p$ is the vector valued quadratic form
$II_{p}:T_{p}S\rightarrow N_{p}S$ given by
$$II_{p}(\textbf{w})=(l_{1}w_{1}^{2}+2m_{1}w_{1}w_{2}+n_{1}w_{2}^{2})\textbf{e}_{3}+(l_{2}w_{1}^{2}+2m_{2}w_{1}w_{2}+n_{2}w_{2}^{2})\textbf{e}_{4},$$
where $l_{i}=\langle f_{xx},\textbf{e}_{i+2}\rangle,\ m_{i}=\langle f_{xy},\textbf{e}_{i+2}\rangle$
and $n_{i}=\langle f_{yy},\textbf{e}_{i+2}\rangle$ for $i=1,2$ are
called the coefficients of the second fundamental form with respect to the
frame above and $\textbf{w}=w_{1}\textbf{e}_{1}+w_{2}\textbf{e}_{2}\in T_{p}S$. The matrix of the second fundamental
form with respect to the orthonormal frame above is given by
$$
\alpha=\left(
         \begin{array}{ccc}
           l_{1} & m_{1} & n_{1} \\
           l_{2} & m_{2} & n_{2} \\
         \end{array}
       \right).
$$

The \emph{resultant} of the quadratic forms is a scalar invariant of the surface defined by Little in \cite{Little}, given by
$$\delta=\frac{1}{4}(4(l_{1}m_{2}-m_{1}n_{2})(m_{1}n_{2}-n_{1}m_{2})-(l_{1}n_{2}-n_{1}l_{2})^{2}).$$
A point $p\in S$ is hyperbolic or elliptic according to whether $\delta(p)$ is negative
or positive, respectively. If $\delta(p)$ is equal to zero, the point is parabolic or an inflection, according to the rank of $\alpha$: $p$ is parabolic if the rank is $2$ and an inflection if it is less than $2$.

A non zero tangent direction $\textbf{u}\in T_{p}S$ is an \emph{asymptotic direction} if there is a non zero vector $v\in N_{p}M$ such that
$$\langle II(\textbf{u},\textbf{w}),v\rangle=0,\ \ \forall\ \textbf{w}\in T_{p}S.$$
Furthermore, $v\in N_{p}S$ is a \emph{binormal direction}.

One can obtain a lot of geometrical information of a regular surface $S\subset\mathbb{R}^{4}$, by studying the generic contact of the surface with hyperplanes. Such contact is measured
by the singularities of the height function of $S$. Let $f:U\rightarrow\mathbb{R}^{4}$ be a local parametrisation of $S$. The family of height functions is given by
$$H:U\times \mathbb{S}^{3}\rightarrow\mathbb{R},\ \  H(u,v)=\langle f(u),v\rangle.$$
Fixing $v\in\mathbb{S}^{3}$, the height function $h_{v}$ of $S$ is given by $h_{v}(u)=H(u,v)$ and has the following property: a normal direction $v$ at $p=f(u)\in S$ is a binormal direction if and only if any tangent direction lying in the kernel of the Hessian of $h_{v}$ at $u$ is an asymptotic direction of $S$ at $p$.

\begin{definition}
The \emph{canal hypersurface} of the surface $S\subset\mathbb{R}^{4}$ is the $3$-manifold
$$CS(\varepsilon)=\{p+\varepsilon v\in\mathbb{R}^{4}|\ p\in S\ \mbox{and}\ v\in (N_{p}S)_{1}\}$$
where $(N_{p}S)_{1}$ denotes the unit sphere in $N_{p}S$ and $\varepsilon$ is a small positive real number.
\end{definition}

It is possible to consider $(N_{p}S)_{1}$ as a subset of $\mathbb{S}^{3}$ and as a consequence, identify $(p,v)$ and $p+\varepsilon v$.

We shall denote the family of height functions on $CS(\varepsilon)$ by $\bar{H}:CS(\varepsilon)\times\mathbb{S}^{3}\rightarrow\mathbb{R}$. So, given $w\in\mathbb{S}^{3}$, the height function of $CS(\varepsilon)$ along $w$ is given by $\bar{h}_{w}:CS(\varepsilon)\rightarrow\mathbb{R}$, where $\bar{h}_{w}(p,v)=\bar{H}((p,v),w)$. Given a point $p\in M$, it is a singular point of $h_{v}$ if and only if $(p,v)\in CS(\varepsilon)$ is a singular point of $\bar{h}_{v}$.

The \emph{Gauss map} of the canal hypersurface $CS(\varepsilon)$, $G:CS(\varepsilon)\rightarrow\mathbb{S}^{3}$, is given by $G(p,v)=v$.
Let $K_{c}:CS(\varepsilon)\rightarrow\mathbb{R}$ be the Gauss-Kronecker curvature function of $CS(\varepsilon)$. Then, the singular set of $G$ is the parabolic set
$$K^{-1}_{c}(0)=\{p+\varepsilon v\in CS(\varepsilon)|\ h_{v}\ \mbox{has a degenerate singularity at p}\}$$
of $CS(\varepsilon)$, which is a regular surface except at a finite number of singular points corresponding to the $D_{4}^{\pm}$-singularities oh $\bar{h}_{v}$. The regular part has regular curves corresponding to the cuspidal edge points and those curves may have especial isolated points which are the swallowtail points.



One can characterise geometrically the degenerate singularities of generic height functions. Denote by $\gamma$ the normal section of the surface $S$ tangent to the asymptotic direction $\theta$ at $p$ associated to the binormal direction $v$.

\begin{teo}\cite{Livro}\label{teo-h-hiper}
Let $p$ be a hyperbolic point on a height function generic surface $M\subset\mathbb{R}^{4}$. Then,
\begin{itemize}
    \item[(i)] $p$ is an $A_{2}$ singularity of $h_{v}$ if and only if $\gamma$ has a non vanishing normal torsion at $p$.
    \item[(ii)] $p$ is an $A_{3}$ singularity of $h_{v}$ if and only if $\gamma$ has a vanishing torsion at $p$ and the direction $\theta$ is transversal to the curve of cuspidal edges points of the Gauss map.
\end{itemize}
\end{teo}

A characterisation of the singularities of the height functions at a parabolic point can also be done.

\begin{teo}\cite{Livro}\label{teo-h-parabolic}
Let $M$ be a height function generic surface in $\mathbb{R}^{4}$ and $p\in M$. Suppose $p$ is a parabolic point, but not an inflection point. Then,
\begin{itemize}
    \item[(i)] p is an $A_{2}$-singularity of $h_{v}$ if and only if $\theta$ is transversal to the parabolic curve $\delta$.
    \item[(ii)] p is an $A_{3}$-singularity of $h_{v}$ if and only if $\theta$ is tangent to the parabolic curve $\delta$ with first order contact.
\end{itemize}
\end{teo}


\section{Corank $1$ surfaces in $\mathbb{R}^{4}$}\label{section-notation}

\subsection{The curvature parabola} Here we present a brief study of the differential geometry of corank $1$ surfaces in $\mathbb{R}^{4}$ which can be found in \cite{Benedini/Sinha/Ruas}.
Let $M$ be a corank $1$ surface in $\mathbb{R}^{4}$ at $p$. We take $M$ as the image of a smooth map $g:\tilde{M}\rightarrow \mathbb{R}^{4}$, where $\tilde{M}$ is a smooth regular surface and $q\in\tilde{M}$ is a corank $1$ point of $g$ such that $g(q)=p$. Also, we consider $\phi:U\rightarrow\mathbb{R}^{2}$ a local coordinate system defined in an open neighbourhood $U$ of $q$ at $\tilde{M}$, and by doing this we may consider a local parametrisation $f=g\circ\phi^{-1}$ of $M$ at $p$ (see the diagram below).

$$
\xymatrix{
\mathbb{R}^{2}\ar@/_0.7cm/[rr]^-{f} & U\subset\tilde{M}\ar[r]^-{g}\ar[l]_-{\phi} & M\subset\mathbb{R}^{4}
}
$$

The \emph{tangent line} of $M$ at $p$, $T_{p}M$, is given by $\mbox{Im}\ dg_{q}$, where $dg_{q}:T_{q}\tilde{M}\rightarrow T_{p}\mathbb{R}^{4}$ is the differential map of $g$ at $q$. Hence, the \emph{normal hyperplane} of $M$ at $p$, $N_{p}M$, is the subspace satisfying $T_{p}M\oplus N_{p}M=T_{p}\mathbb{R}^{4}$.

Consider the orthogonal projection $\perp:T_{p}\mathbb{R}^{4}\rightarrow N_{p}M$, $w\mapsto w^{\perp}$.
The \emph{first fundamental form} of $M$ at $p$, $I:T_{q}\tilde{M}\times T_{q}\tilde{M}\rightarrow \mathbb{R}$ is given by
$$I(\textbf{u},\textbf{v})=\langle dg_{q}(\textbf{u}),dg_{q}(\textbf{v})\rangle,\ \ \ \ \forall\ \textbf{u},\textbf{v}\in T_{q}\tilde{M}.$$
Since the map $g$ has corank $1$ at $q\in T_{q}\tilde{M}$, the first fundamental form is not a Riemannian metric on $T_{q}\tilde{M}$, but a pseudometric. Considering the local parametrisation of $M$ at $p$, $f=g\circ\phi^{-1}$ and the basis $\{\partial_{x},\partial_{y}\}$ of $T_{q}\tilde{M}$, the coefficients of the first fundamental form with respect to $\phi$ are:
$$\begin{array}{c}
     E(q)=I(\partial_{x},\partial_{x})=\langle f_{x},f_{x}\rangle(\phi(q)),\   F(q)=I(\partial_{x},\partial_{y})=\langle f_{x},f_{y}\rangle(\phi(q)), \\
      G(q)=I(\partial_{y},\partial_{y})=\langle f_{y},f_{y}\rangle(\phi(q)).
\end{array}$$
Taking $\textbf{u}=\alpha\partial_{x}+\beta\partial_{y}=(\alpha,\beta)\in T_{q}\tilde{M}$, we write $I(\textbf{u},\textbf{u})=\alpha^{2}E(q)+2\alpha\beta F(q)+\beta^{2}G(q)$.

With the same conditions as above, the \emph{second fundamental form} of $M$ at $p$, $II:T_{q}\tilde{M}\times T_{q}\tilde{M}\rightarrow N_{p}M$ in the basis $\{\partial_{x},\partial_{y}\}$ of $T_{q}\tilde{M}$ is given by
$$
\begin{array}{c}
     II(\partial_{x},\partial_{x})=f_{xx}^{\perp}(\phi(q)),\  II(\partial_{x},\partial_{y})=f_{xy}^{\perp}(\phi(q)),\
      II(\partial_{y},\partial_{y})=f_{yy}^{\perp}(\phi(q))
\end{array}
$$
and we extend it to the whole space in a unique way as a symmetric bilinear map. It is possible to show that the second fundamental form does not depend on the choice of local coordinates on $\tilde{M}$.

For each normal vector $\nu\in N_{p}M$, the \emph{second fundamental form along $\nu$}, $II_{\nu}:T_{q}\tilde{M}\times T_{q}\tilde{M}\rightarrow\mathbb{R}$ is given by $II_{\nu}(\textbf{u},\textbf{v})=\langle II(\textbf{u},\textbf{v}),\nu\rangle$, for all $\textbf{u},\textbf{v}\in T_{q}\tilde{M}$. The coefficients of $II_{\nu}$ with respect to the basis $\{\partial_{x},\partial_{y}\}$ of $T_{q}\tilde{M}$ are
$$
\begin{array}{cc}
     l_{\nu}(q)=\langle f_{xx}^{\perp},\nu\rangle(\phi(q)),\ m_{\nu}(q)=\langle f_{xy}^{\perp},\nu\rangle(\phi(q)),  \\
     n_{\nu}(q)=\langle f_{yy}^{\perp},\nu\rangle(\phi(q)).
\end{array}
$$

Fixing an orthonormal frame $\{\nu_{1},\nu_{2},\nu_{3}\}$ of $N_{p}M$,
$$
\begin{array}{cl}\label{eq.2ff}
II(\textbf{u},\textbf{u}) & =II_{\nu_{1}}(\textbf{u},\textbf{u})\nu_{1}+II_{\nu_{2}}(\textbf{u},\textbf{u})\nu_{2}+II_{\nu_{3}}(\textbf{u},\textbf{u})\nu_{3} \\
        & =\displaystyle{\sum_{i=1}^{3}(\alpha^{2}l_{\nu_{i}}(q)+2\alpha\beta m_{\nu_{i}}(q)+\beta^{2}n_{\nu_{i}}(q))\nu_{i}},
\end{array}
$$
Moreover, the second fundamental form is represented by the matrix of coefficients
$$
\left(
  \begin{array}{ccc}
    l_{\nu_{1}} & m_{\nu_{1}} & n_{\nu_{1}} \\
    l_{\nu_{2}} & m_{\nu_{2}} & n_{\nu_{2}} \\
    l_{\nu_{3}} & m_{\nu_{3}} & n_{\nu_{3}} \\
  \end{array}
\right).
$$

\begin{definition}\cite{Benedini/Sinha/Ruas}\label{curvatureparabola}
Let $C_{q}\subset T_{q}\tilde{M}$ be the subset of unit tangent vectors and let $\eta_{q}:C_{q}\rightarrow N_{p}M$ be the map given by $\eta_{q}(\textbf{u})=II(\textbf{u},\textbf{u})$. The \emph{curvature parabola} of $M$ at $p$, denoted by $\Delta_{p}$, is the image of $\eta_{q}$, that is, $\eta_{q}(C_{q})$.
\end{definition}

The curvature parabola is a plane curve whose trace lies in the normal hyperplane of the surface. Also, this curve may degenerate into a half-line, a line or even a point.

\begin{ex}\label{ex}
Consider $\tilde{M}=\mathbb{R}^{2}$ and the singular surface $M$ locally parametrised by the $I_{1}$-singularity $f(x,y)=(x,xy,y^{2},y^{3})$. Taking coordinates $(X,Y,Z,W)$ in $\mathbb{R}^{4}$, $q=(0,0)$ and $p=(0,0,0,0)$, the tangent line $T_{p}M$ is the $X$-axis and $N_{p}M$ is the $YZW$-hyperplane. The coefficients of the first fundamental form are given by $E(q)=1$ and $F(q)=G(q)=0$. Hence, if $\textbf{u}=(\alpha,\beta)\in T_{q}\mathbb{R}^{2}$, $I(\textbf{u},\textbf{u})=\alpha^{2}$ and $C_{q}=\{(\pm1,y):y\in\mathbb{R}\}$. The matrix of coefficients of the second fundamental form is
$$\left(
\begin{array}{ccc}
    0 & 1 & 0  \\
    0 & 0 & 2  \\
    0 & 0 & 0  \\
\end{array}
\right)
$$
when we consider the orthonormal frame $\{\textbf{e}_{1},\textbf{e}_{2},\textbf{e}_{3},\textbf{e}_{4}\}$. Therefore, for $\textbf{u}=(\alpha,\beta)$, $II(\textbf{u},\textbf{u})=(0,2\alpha\beta,2\beta^{2},0)$ and
the curvature parabola $\Delta_{p}$ is a non-degenerate parabola which can be parametrised by $\eta(y)=(0,2y,2y^{2},0)$.
\end{ex}


\subsection{Second order properties}
Given a regular surface $N\subset\mathbb{R}^{5}$, we
consider the corank $1$ surface $M$ at $p$ obtained by the
projection of $N$ in a tangent direction, via the map
$\xi:N\subset\mathbb{R}^{5}\rightarrow M$. The regular surface
$N\subset\mathbb{R}^{5}$ can be taken, locally, as the image of an
immersion $i:\tilde{M}\rightarrow N\subset\mathbb{R}^{5}$, where
$\tilde{M}$ is the regular surface from the construction done
before.

The points of $N$ can be
characterized according to the rank of its fundamental form at that
point. Inspired by this classification, we have the following:

\begin{definition}
Given a corank $1$ surface $M\subset\mathbb{R}^{4}$, we define the subset
$$M_{i}=\{p\in M: p\ \mbox{is singular and}\ rank(II_{p})=i\},\ i=0,1,2,3.$$
\end{definition}

\begin{definition}
The minimal affine space which contains the curvature parabola is denoted by $\mathcal{A}ff_{p}$. The plane denoted by $E_{p}$ is the vector space: parallel to $\mathcal{A}ff_{p}$ when $\Delta_{p}$ is a non degenerate parabola, the plane through $p$ that contains $\mathcal{A}ff_{p}$ when $\Delta_{p}$ is a non radial half-line or a non radial line and any plane through $p$ that contains $\mathcal{A}ff_{p}$ when $\Delta_{p}$ is a radial half-line, a radial line or a point.
\end{definition}


Let $S\subset\mathbb{R}^{4}$ be the regular surface locally obtained by projecting $N\subset\mathbb{R}^{5}$ via the map $\pi$ into the four space given by $T_{\xi^{-1}(p)}N\oplus \xi^{-1}(E_{p})$ (see the following diagram).

$$
\xymatrix{
 &  & N\subset\mathbb{R}^{5}\ar[rd]^-{\pi}\ar[d]^-{\xi} & \\
 \mathbb{R}^{2}\ar@/_0.7cm/[rr]^-{f}& \tilde{M}\ar[l]_-{\phi}\ar[r]^-{g}\ar[ru]^-{i}& M\subset\mathbb{R}^{4}& S\subset\mathbb{R}^{4}
}
$$


Using the previous construction, one can relate the corank $1$ singular surface $M\subset\mathbb{R}^{4}$ and the regular surface $S\subset\mathbb{R}^{4}$.

\begin{definition}
A non zero direction $\textbf{u}\in T_{q}\tilde{M}$ is called \emph{asymptotic} if there is a non zero vector $\nu\in E_{p}$ such that
$$II_{\nu}(\textbf{u},\textbf{v})=\langle II(\textbf{u},\textbf{v}),\nu\rangle=0\ \ \forall\ \textbf{v}\in T_{q}\tilde{M}.$$
Moreover, in such case, we say that $\nu$ is a \emph{binormal direction}.
\end{definition}

The normal vectors $\nu\in N_{p}M$ satisfying the condition
$II_{\nu}(\textbf{u},\textbf{v})=0$ are called \emph{degenerate
directions}, but only those in $E_p$ are binormal directions.
When $p\in M_{1}\cup M_{0}$, the choice of $E_{p}$ does not change the number of binormal directions. Furthermore, all directions $\textbf{u}\in T_{q}\tilde{M}$ are asymptotic.

\begin{definition}
Given a binormal direction $\nu\in E_{p}$, the hyperplane through $p$ and orthogonal to $\nu$ is called an \emph{osculating hyperplane} to $M$ at $p$.
\end{definition}

\begin{definition}\label{pointstypes}
Given a surface $M\subset\mathbb{R}^{4}$ with corank $1$ singularity at $p\in M$. The point $p$ is called:
\begin{itemize}
    \item[(i)] \emph{elliptic} if there are no asymptotic directions at $p$;
    \item[(ii)] \emph{hyperbolic} if there are two asymptotic directions at $p$;
    \item[(iii)] \emph{parabolic} if there is one asymptotic direction at $p$;
    \item[(iv)] \emph{inflection} if there are an infinite number of asymptotic directions at $p$.
\end{itemize}
\end{definition}

The next result compares the geometry of a corank $1$ surface in $\mathbb{R}^{4}$ with the geometry of the associated regular surface $S\subset\mathbb{R}^{4}$ obtained.

\begin{teo}\label{teorelation}\cite{Benedini/Sinha/Ruas}
Let $M\subset\mathbb{R}^{4}$ be a surface with corank $1$
singularity at $p\in M$ and $S\subset\mathbb{R}^{4}$ the regular
surface associated to $M$.
\begin{itemize}
    \item[(i)] A direction $\textbf{u}\in T_{q}\tilde{M}$ is an asymptotic direction of $M$ if and only if it is also an asymptotic direction of the associated regular surface $S\subset\mathbb{R}^{4}$;
   \item[(ii)] A direction $\nu\in N_{p}M$ is a binormal direction of $M$ if and only if $\pi\circ
   \xi^{-1}(\nu)\in N_{\pi\circ\xi^{-1}(p)}S$ is a binormal direction of $S$.
    \item[(iii)] The point $p$ is an elliptic/hyperbolic/parabolic/inflection point if and only if $\pi\circ\xi^{-1}(p)\in S$ is an elliptic/hyperbolic/parabolic/inflection point, respectively.
\end{itemize}
\end{teo}

The singularity $I_{k}$, $k\geqslant1$, given by the $\mathcal{A}$-normal form $(x,y)\mapsto(x,xy,y^{2},y^{2k+1})$ has an interesting property: every map germ $\mathcal{A}$-equivalent to it prarametrises a corank $1$ surface in $\mathbb{R}^{4}$ whose curvature parabola is a non degenerate parabola. Moreover, $I_{k}$ are the only singularities having this property. Hence, every map germ $\mathcal{A}$-equivalent to $I_{k}$ is $\mathcal{R}^{2}\times\mathcal{O}(4)$-equivalent to the normal form
$f:(\mathbb{R}^{2},0)\rightarrow(\mathbb{R}^{4},0)$ where $$f(x,y)=(x,xy+p(x,y),b_{20}x^{2}+b_{11}xy+b_{02}y^{2}+q(x,y),c_{20}x^{2}+r(x,y))$$
with $b_{02}>0$ and $p,q,r\in\mathcal{M}_{2}^{3}$. The proof of this assertion can be found in \cite{Benedini/Sinha/Ruas}.

\begin{prop}\cite{Benedini/Sinha/Ruas}
Consider the $\mathcal{R}^{2}\times\mathcal{O}(4)$ normal form of the singularity $I_{k}$ given above. Then, the singularity $I_{k}$ is hyperbolic, parabolic or elliptic if and only if $b_{20}$ is positive, zero or negative, respectively.
\end{prop}

For corank $1$ surfaces in $\mathbb{R}^{4}$ we have the following:

\begin{definition}\label{def.curvatura}
The non-negative number
$$\kappa_{u}(p)=d(p,\mathcal{A}ff_{p})$$
is called the \emph{umbilic curvature} of $M$ at $p$.
\end{definition}

The authors in \cite{Benedini/Sinha/Ruas} present explicit formulas of this invariant as well as geometric interpretations of it. Here, however, we shall restrict our study to the case where $\Delta_{p}$ is a non degenerate parabola.

\begin{prop}\cite{Benedini/Sinha/Ruas}
Let $\{\nu_{1},\nu_{2},\nu_{3}\}$ be an othonormal frame of $N_{p}M$ such that $E_{p}=\{\nu_{1},\nu_{2}\}$ and $E_{p}^{\perp}=\{\nu_{3}\}$. Then the following holds:
$$\kappa_{u}(p)=\frac{|II_{\nu_{3}}(\textbf{u},\textbf{u})|}{I(\textbf{u},\textbf{u})}=|\mbox{proj}_{\nu_{3}}\eta(y)|=|\langle \eta(y),\nu_{3}\rangle|,$$
for any $\textbf{u}\in T_{q}\tilde{M}$, where $\eta$ is a parametrisation of $\Delta_{p}$.
\end{prop}



\section{Flat geometry}

In this section we study the contact of a singular surface $M\subset\mathbb{R}^{4}$ locally given by the $\mathcal{A}$-normal form $(x,y)\mapsto(x,xy,y^{2},y^{3})$ with hyperplanes. One can summarize the \emph{modus operandi} in the following way: we fix a model of the singularity $I_{1}$ and study the contact
with the zero fibres of submersions. We then associate the singularities of the height functions with the geometry studied in the previous section.

\subsection{Functions on $I_{1}$}

In this section, we classify germs of functions on $\texttt{X}\subset\mathbb{R}^{4}$, where $\texttt{X}$ is the germ of the model surface locally parametrised by the $I_{1}$ singularity. This technique was introduced in \cite{BruceWest}, where the authors study the contact between the Whitney umbrella (or crosscap) with planes. More recently, the same was done in \cite{OsetSinhaTari2} and \cite{Oset/Saji} but this time the surfaces were the cuspidal edge and the folded umbrella, respectively.

We denote by $\mathcal{E}_{n}$ the local ring of germs of functions $f:(\mathbb{R}^{n},0)\rightarrow\mathbb{R}$ and by $\mathcal{M}_{n}$ its maximal ideal.
Let $(\texttt{X},0)\subset(\mathbb{R}^{n},0)$ be a germ of a reduced analytic subvariety of $\mathbb{R}^{n}$ at $0$ defined by an ideal $I$ of $\mathcal{E}_{n}$. A diffeomorphism $k:(\mathbb{R}^{n},0)\rightarrow(\mathbb{R}^{n},0)$ is said to preserve $\texttt{X}$ if $(k(\texttt{X},0))=(\texttt{X},0)$.
The group of such diffeomorphisms is a subgroup of the group $\mathcal{R}$ and is denoted by $\mathcal{R}(\texttt{X})$. This is one of the Damon's ``geometrical subgroups" of $\mathcal{K}$ (see \cite{Bruce/Roberts,Damon}).

Consider the $\mathcal{A}$-normal form of the $I_{1}$ singularity: $f(x,y)=(x,xy,y^{2},y^{3})$. Our aim is to classify germs of submersions $g:(\mathbb{R}^{4},0)\rightarrow(\mathbb{R},0)$ using the $\mathcal{R}(\texttt{X})$ equivalence, where $\texttt{X}=f(\mathbb{R}^{2},0)$ is our model surface. The ideal $I\lhd\mathcal{E}_{4}$ of irreducible polynomials defining $\texttt{X}$ is given by
$$I=\langle Y^{2}-X^{2}Z,\ W^{2}-Z^{3},\ XW-YZ,\ YW-XZ^{2} \rangle.$$

We shall denote by $\Theta(\texttt{X})$ the $\mathcal{E}_{4}$-module of vector fields tangent to $\texttt{X}$ (Derlog$(\texttt{X})$ in other texts). Hence, we have
$$\xi\in\Theta(\texttt{X}) \Leftrightarrow\xi h(x)=dh_{x}(\xi(x))\in I,\ \forall h\in I.$$

\begin{prop}\label{derlog}
$\Theta(\texttt{X})$ is generated by:
$$\begin{array}{ll}\vspace{0.3cm}
\xi_{1}=X\frac{\partial}{\partial X}+Y\frac{\partial}{\partial Y},&\ \xi_{2}=X^{2}\frac{\partial}{\partial X}+2Y\frac{\partial}{\partial Z}+3XZ\frac{\partial}{\partial W}, \\ \vspace{0.3cm}
\xi_{3}=Y\frac{\partial}{\partial Y}+2Z\frac{\partial}{\partial Z}+3W\frac{\partial}{\partial W},&\ \xi_{4}=Y\frac{\partial}{\partial X}+XZ\frac{\partial}{\partial Y},\\ \vspace{0.3cm}
\xi_{5}=Z\frac{\partial}{\partial X}+W\frac{\partial}{\partial Y},&\ \xi_{6}=XZ\frac{\partial}{\partial Y}+2W\frac{\partial}{\partial Z}+3Z^{2}\frac{\partial}{\partial W},\\ \vspace{0.3cm}
\xi_{7}=W\frac{\partial}{\partial X}+Z^{2}\frac{\partial}{\partial Y},&\ \xi_{8}=(Y^{2}-X^{2}Z)\frac{\partial}{\partial W},\\ \vspace{0.3cm}
\xi_{9}=(YZ-XW)\frac{\partial}{\partial W},&\ \xi_{10}=XW\frac{\partial}{\partial Y}+2Z^{2}\frac{\partial}{\partial Z}+3ZW\frac{\partial}{\partial W},\\ \vspace{0.3cm}
\xi_{11}=(YW-XZ^{2})\frac{\partial}{\partial W},&\ \xi_{12}=(W^{2}-Z^{3})\frac{\partial}{\partial W},\\ \vspace{0.3cm}
\xi_{13}=(W^{2}-Z^{3})\frac{\partial}{\partial Y}.\vspace{0.3cm}
\end{array}$$
\end{prop}
\begin{proof}
For notation purposes we write $(X,Y,Z,W)=(X_1,X_2,X_3,X_4)$. We are looking for vector fields $\xi=\sum_{i=1}^4\xi_i\frac{\partial}{\partial X_i}\in \theta_4$ such that for each $j=1,\ldots,4$ there exist functions $\alpha_i(X_1,\ldots,X_4)$ such that $$\sum_{i=1}^4\xi_i\frac{\partial h_j}{\partial X_i}=\sum_{i=1}^4\alpha_ih_i.$$ Consider, for $j=1,\ldots,4$, the map $\Phi_j:\mathcal E_4^8\rightarrow \mathbb R$ given by $$\Phi_j(\xi,\alpha)=\sum_{i=1}^4\xi_i\frac{\partial h_j}{\partial X_i}-\sum_{i=1}^4\alpha_ih_i,$$ where $\xi=(\xi_1,\ldots,\xi_4)\in\mathcal E_4^4$ and $\alpha=(\alpha_1,\ldots,\alpha_4)\in\mathcal E_4^4$. Let $A_j=\ker\Phi_j$. Let $\pi:\mathcal E_4^8\rightarrow \mathcal E_4^4$ be the canonical projection given by $\pi(\xi,\alpha)=\xi$. Let $B_j=\pi(A_j)$. Then $$\Theta(\texttt{X})=\bigcap_{j=1}^4B_j.$$

In order to obtain the $A_j$ we use syzygies in the computer package Singular. It can be checked that all the vector fields obtained by this method are, in fact, liftable, i.e. there exists a vector field $\eta\in\theta_2$ such that $dh(\eta)=\xi\circ h$, and are therefore tangent to $\texttt{X}$.
\end{proof}

The idea for classifying analytic function germs $g:(\mathbb{R}^{4},0)\rightarrow(\mathbb{R},0)$ up to $\mathcal{R}(\texttt{X})$-equivalence is to use generalisations of the standard results for the group $\mathcal{R}$, that is, when $\texttt{X}=\emptyset$. Since $\mathcal{R}(\texttt{X})$ is one of the Damon's ``geometrical subgroups" of $\mathcal{K}$, there are versions of the unfolding and determinacy theorems. In this classification, the orbits are obtained inductively on the jet level and the complete transversal method is also adapted for our action.

We define $\Theta_{1}(\texttt{X})=\{\xi\in\Theta(\texttt{X}):j^{1}\xi=0\}$. Hence, from Proposition \ref{derlog},
$$\Theta_{1}(\texttt{X})=\mathcal{M}_{4}\{\xi_{1}\ldots,\xi_{7}\}+\mathcal{E}_{4}\{\xi_{8},\ldots,\xi_{13}\}.$$
For each $f\in\mathcal{E}_{4}$, $\Theta(\texttt{X})\cdot f=\{\xi(f):\xi\in\Theta(\texttt{X})\}$. A similar definition is made for $\Theta_{1}(\texttt{X})\cdot f$. Furthermore, we define the \emph{tangent spaces} to the $\mathcal{R}(\texttt{X})$-orbit of $f$:
$$L\mathcal{R}_{1}(\texttt{X})\cdot f=\Theta_{1}(\texttt{X})\cdot f,\ L\mathcal{R}(\texttt{X})\cdot f=L\mathcal{R}_{e}(\texttt{X})\cdot f=\Theta(\texttt{X})\cdot f.$$

The $\mathcal{R}(\texttt{X})$-codimension is given by $d(f,\mathcal{R}(\texttt{X}))=\dim_{\mathbb{R}}(\mathcal{E}_{4}/L\mathcal{R}(\texttt{X})\cdot f)$.

\begin{prop}\cite{BruceWest}\label{transversal}
Let $f:(\mathbb{R}^{4},0)\rightarrow(\mathbb{R},0)$ be a smooth germ and $h_{1},\ldots,h_{r}$ be homogeneous polynomials of degree $k+1$ with the property that
$$\mathcal{M}_{4}^{k+1}\subset L\mathcal{R}_{1}(\texttt{X})\cdot f+sp\{h_{1},\ldots,h_{r}\}+\mathcal{M}_{4}^{k+2}.$$
Then any germ $g$ with $j^{k}f(0)=j^{k}g(0)$ is $\mathcal{R}_{1}(\texttt{X})$-equivalent to a germ of the form $f+\sum_{i=1}^{r}u_{i}h_{i}+\phi$, where $\phi\in \mathcal{M}_{4}^{k+2}$. The vector subspace $sp\{h_{1},\ldots,h_{r}\}$ is called a complete $(k+1)-\mathcal{R}(\texttt{X})$-transversal of $f$.
\end{prop}

\begin{coro}\cite{BruceWest}\label{lema-finit.det.}
The following hold:
\begin{itemize}
\item [(i)] If $\Theta_{1}(\texttt{X})\cdot f+\mathcal{M}_{4}^{k+2}\supset\mathcal{M}_{4}^{k+1}$,
then $f$ is $k\mbox{-}\mathcal{R}(\texttt{X})$-determined;
\item[(ii)] If every vector field in $\Theta(\texttt{X})$ vanishes at the origin and
$\Theta(\texttt{X})\cdot f+\mathcal{M}_{4}^{k+2}\supset\mathcal{M}_{4}^{k+1}$, then
$f$ is $(k+1)\mbox{-}\mathcal{R}(\texttt{X})$-determined.
\end{itemize}
\end{coro}

The next result about trivial families will be needed.

\begin{prop}\cite{BruceWest}\label{trivial}
Let $F:(\mathbb{R}^{4}\times\mathbb{R},(0,0))\rightarrow(\mathbb{R},0)$ be a smooth family of functions such that
$F(0,t)=0$ for $t$ small enough. Also, let $\xi_{1},\ldots,\xi_{p}$ be vector fields in $\Theta(\texttt{X})$ that vanish at the origin. Then, the family $F$ is $k\mbox{-}\mathcal{R}(\texttt{X})$-trivial if
$\frac{\partial F}{\partial t}\in \langle\xi_{1}(F),\ldots,\xi_{p}(F)\rangle+\mathcal{M}_{4}^{k+1}\mathcal{E}_{5}\subset \mathcal{E}_{5}$.
\end{prop}

Two families of germs of functions $F$ and $G:(\mathbb{R}^{4}\times\mathbb{R}^{a},(0,0))\rightarrow(\mathbb{R},0)$ are $P-\mathcal{R}^{+}(\texttt{X})$-equivalent  if there exist a germ of a diffeomorphism $\Psi:(\mathbb{R}^{4}\times\mathbb{R}^{a},(0,0))\rightarrow(\mathbb{R}^{4}\times\mathbb{R}^{a},(0,0))$ preserving $(\texttt{X}\times\mathbb{R}^{a},(0,0))$ and of the form $\Psi(x,u)=(\alpha(x,u),\psi(x,u))$ and a germ $c:(\mathbb{R}^{a},0)\rightarrow\mathbb{R}$ such that $G(x,u)=F(\Psi(x,u))+c(u)$.

A family $F$ is said to be an $\mathcal{R}^{+}(\texttt{X})$-versal deformation of $F_{0}(x)=F(x, 0)$ if any
other deformation $G$ of $F_{0}$ can be written in the form $G(x,u)=F(\Psi(x,u))+c(u)$ for some germs of smooth mappings $\Psi$ and $c$ as above with $\Psi$ not necessarily a germ
of diffeomorphism.

\begin{prop}\cite{BruceWest}
A deformation $F:(\mathbb{R}^{4}\times\mathbb{R}^{a},(0,0))\rightarrow(\mathbb{R},0)$ of a germ of function $f$ on $\texttt{X}$ is $\mathcal{R}^{+}(\texttt{X})$-versal if and only if
$$L\mathcal{R}_{e}(\texttt{X})\cdot f+\mathbb{R}.\{1,\dot{F}_{1},\ldots,\dot{F}_{a}\}=\mathcal{E}_{4},$$
where $\dot{F}_{i}(x)=\frac{\partial F}{\partial u_{i}}(x,0)$.
\end{prop}

\begin{teo}\label{classification}
Let $\texttt{X}$ be the germ of the $\mathcal{A}$-model surface parametrised by
$f(x,y)=(x,xy,y^{2},y^{3})$.
Then, any germ of a $\mathcal{R}(\texttt{X})$-finitely determined submersion in
$\mathcal{M}_{4}$ with $\mathcal{R}(\texttt{X})$-codimension $\leqslant3$ is $\mathcal{R}(\texttt{X})$-equivalent to one of the germs in Table \ref{orbitas}.

\begin{table}[h]
\caption{Germs of submersions in $\mathcal{M}_{4}$ of $\mathcal{R}(\texttt{X})\mbox{-codimension}\leqslant3$}
\centering{
\begin{tabular}{lcl}
\hline
Normal form & $d(f,\mathcal{R}(\texttt{X}))$ & $\mathcal{R}(\texttt{X})$-versal deformation \\
\hline
$X$ & $0$ & $X$ \\
$\pm Z\pm X^{2}$ & $1$ & $\pm Z\pm X^{2}+a_{1}X$\\
$\pm Z+X^{3}$ & $2$ & $\pm Z+X^{3}+a_{1}X+a_{2}X^{2}$\\
$\pm Z\pm X^{4}$ & $3$ & $\pm Z\pm X^{3}+a_{1}X+a_{2}X^{2}+a_{3}X^{3}$\\ $Y$ & $2$ & $Y+a_{1}X+a_{2}Z$ \\
$\pm W\pm X^{2}$ & $3$ & $\pm W\pm X^{2}+a_{1}X+a_{2}Y+a_{3}Z$\\
\hline
\end{tabular}
}
\label{orbitas}
 \end{table}
\end{teo}
\begin{proof}
We shall consider the vector fields in Proposition \ref{derlog}. The linear change of coordinates in $\mathcal{R}(\texttt{X})$ obtained by integrating the $1$-jets of the vector fields in $\Theta(\texttt{X})$ are:
$$\begin{array}{ll}
\eta_{1}=(e^{\alpha}X,e^{\alpha}Y,Z,W),\ \alpha\in\mathbb{R},&
\eta_{2}=(X,Y,Z+\alpha Y,W),\ \alpha\neq0,\\
\eta_{3}=(X,e^{\alpha}Y,e^{2\alpha}Z,e^{3\alpha}W),\ \alpha\in\mathbb{R},& \
\eta_{4}=(X+\alpha Y,Y,Z,W),\ \alpha\neq0,\\
\eta_{5}=(X+\alpha Z,Y+\alpha W,Z,W),\ \alpha\neq0,& \
\eta_{6}=(X,Y,Z+\alpha W,W),\ \alpha\neq0, \\
\eta_{7}=(X+\alpha W,Y,Z,W),\ \alpha\neq0,& \
\eta_{8}=(-X,-Y,Z,W).
\end{array}
$$
Consider the non zero $1$-jet $g=aX+bY+cZ+dW$. If $a\neq0$, after changes of coordinates ($\eta_{i},\ i=4,5,7,1,8$, in this order) we get $X$. If $a=0\neq c$, (using $\eta_{i},\ i=2,6,3$) we get $\pm Z$. If $a=c=0\neq b$, (using $\eta_{i},\ i=5,1,8$) we have $Y$. At last, if $a=b=c=0\neq d$, using $\eta_{3}$, we have $W$.
\begin{itemize}
    \item[(i)] Consider the $1$-jet $g=X$. This case is the most simple. Notice that every vector field $\xi_{i}\in \Theta(\texttt{X})$ vanishes at the origin and $\mathcal{M}_{4}\subset\Theta(\texttt{X})\cdot g+\mathcal{M}_{4}^{2}$,
so $g$ is $1\mbox{-}\mathcal{R}(\texttt{X})$-determined by Corollary \ref{lema-finit.det.}. Also,
$$\mathcal{R}(\texttt{X})\mbox{-}cod(g)=\dim_{\mathbb{R}}(\mathcal{M}_{4}/\Theta(\texttt{X})\cdot g)=0.$$
\item[(ii)] Consider the $1$-jet $g=\pm Z$. For $k\geqslant2$, the complete $k\mbox{-}\mathcal{R}(\texttt{X})$-transversal of $g$ is given by $\pm Z+\delta X^{k}$. If $\delta\neq0$, $\pm Z+\delta X^{k}\thicksim_{\eta_{1}}g_{k}= \pm Z+(-1)^{k+1} X^{k}$. For $g_{k}$, $\mathcal{M}_{4}^{k}\subset \Theta(\texttt{X})\cdot g_{k}+\mathcal{M}_{4}^{k+1}$, that is, $g_{k}$ is $k\mbox{-}\mathcal{R}(\texttt{X})$-determined and
$\mathcal{R}(\texttt{X})\mbox{-}cod(g_{k})=k-1$.
\item[(iii)] Now, consider the $1$-jet $g=Y$. The complete $2\mbox{-}\mathcal{R}(\texttt{X})$-transversal of $g$ is given by
$g=Y+\beta X^{2}+\gamma Z^{2}+\delta XZ$. Consider $g$ as a $1$-parameter family of germs of functions parametrised by $\gamma$. Then $\partial g/\partial \gamma=Z^{2}\in\langle\xi_{1}(g),\ldots,\xi_{13}(g)\rangle+\mathcal{M}_{4}^{3}$. So, by Proposition \ref{trivial}, $g$ is equivalent to $Y+\beta X^{2}+\delta XZ$. In a similar way, we can prove that considering $g$ a family parametrised by $\delta$ and then by $\beta$, we have $g$ equivalent to $Y$. Moreover, $g=Y$ is $2\mbox{-}\mathcal{R}(\texttt{X})$-determined, since
$\mathcal{M}_{4}^{2}\subset \Theta(\texttt{X})\cdot g+\mathcal{M}_{4}^{3}$ and $\mathcal{R}(\texttt{X})\mbox{-}cod(g)=2$.
\item[(iv)] The last $1$-jet is $g=W$. Now, the complete $2\mbox{-}\mathcal{R}(\texttt{X})$ transversal is $g=\pm W+\alpha X^{2}+\beta Z^{2}+\gamma XY+\delta XZ$. Considering $g$ a a $1$-parameter family of germs of functions parametrised by $\beta$, it is possible to show that it $2\mbox{-}\mathcal{R}(\texttt{X})$-trivial and so $g$ is equivalent to $\pm W+\alpha X^{2}+\gamma XY+\delta XZ$. At this point, we split the study in two cases. If $\alpha\neq0$, using again the triviality result, we show that the germ is equivalent to $\pm W+\alpha X^{2}\sim_{\eta_{1}}\pm W\pm X^{2}$. Besides, $g$ is now $2\mbox{-}\mathcal{R}(\texttt{X})$-determined and $\mathcal{R}(\texttt{X})\mbox{-}cod(g)=3$. However, when $\alpha=0$, the germs obtained have stratum codimension greater than $3$ and will not be considered here.
\end{itemize}
Therefore, we conclude the proof.
\end{proof}

\subsection{Contact with hyperplanes}

The following result gives us a generic normal form up to order $3$ for any surface whose local parametrisation is $\mathcal{A}$-equivalent to the singularity $I_{1}$.

\begin{teo}\label{teo.normal-form}
Let $f_{1}:(\mathbb{R}^{2},0)\rightarrow(\mathbb{R}^{4},0)$ be a map germ $\mathcal{A}$-equivalent to $f(x,y)=(x,xy,y^{2},y^{3})$.
Then, there are smooth change of coordinates in the source and isometries in the target that make $f_{1}$ equivalent to
$$\left(x,xy+h(y),\sum_{i+j=2,3}b_{ij}x^{i}y^{j},c_{20}x^{2}+\sum_{i+j=3}c_{ij}x^{i}y^{j}\right)+o(4),$$
with $b_{ij},c_{ij}\in\mathbb{R}$, $h\in\mathcal{M}^{4}$ and $b_{02} c_{03}\neq0$.
\end{teo}
\begin{proof}
In \cite{Benedini/Sinha/Ruas}, is proved that $I_{1}$ is $\mathcal{R}^{2}\times\mathcal{O}(4)$-equivalent to
$$(x,y)\mapsto\left(x,xy+a_{03}y^{3},\sum_{i+j=2,3}b_{ij}x^{i}y^{j},c_{20}x^{2}+\sum_{i+j=3}c_{ij}x^{i}y^{j}\right)+o(4),$$
com $b_{02},c_{03}\neq0$. In order to obtain the desired normal form, we have to eliminate $a_{03}y^{3}$. Consider the change $T$ and the angle $\theta=\mbox{arctan}(a_{03}/c_{03})$, such that $(\sin\theta,\cos\theta)=(a_{03},c_{03})/\sqrt{a_{03}^{2}+c_{03}^{2}}$:
$$T=\left(
      \begin{array}{cccc}
        1 & 0 & 0 & 0 \\
        0 & \cos\theta & 0 & -\sin\theta \\
        0 & 0 & 1 & 0 \\
        0 & \sin\theta & 0 & \cos\theta \\
      \end{array}
    \right).
$$
Hence, we obtain
$$\left(x,\cos\theta xy-\sin\theta(c_{20}x^{2}+c_{30}x^{3}+c_{21}x^{2}y+c_{12}xy^{2}),\sum_{i+j=2,3}b_{ij}x^{i}y^{j},\bar{c}_{20}x^{2}+\sum_{i+j=3}\bar{c}_{ij}x^{i}y^{j}\right).$$
To eliminate the monomials $x^{2},x^{3},x^{2}y$ and $xy^{2}$ from the second coordinate, take the change in the source given by:
$$x\mapsto x'=x\ \mbox{e}\ y\mapsto y'=y+\frac{\sin\theta}{\cos\theta}(c_{20}x+c_{30}x^{2}+c_{21}xy+c_{12}y^{2}).$$
Therefore, we have
$$\left(x,\cos\theta xy,\sum_{i+j=2,3}a_{ij}x^{i}y^{j},\bar{c}_{20}x^{2}+\sum_{i+j=3}\bar{c}_{ij}x^{i}y^{j}\right)+o(4).$$
Finally, a change of coordinates in the source provides the generic normal form.
\end{proof}

Given a corank $1$ surface $M\subset\mathbb{R}^{4}$ at $p$, locally parametrised by the normal form in Theorem \ref{teo.normal-form}, we can deduce some information: The plane $E_{p}$ is the $YZ$-plane, the umbilic curvature is given by $\kappa_{u}(p)=2|c_{20}|$ and the tangent cone $C_{p}M$ is the $XZ$-plane.


Let $M\subset\mathbb{R}^{4}$ be a corank $1$ surface locally parametrised by a map germ $\mathcal{A}$-equivalent to $I_{1}$. The \emph{family of height functions} of $M$ is given by
$$H:M\times\mathbb{S}^{3}\rightarrow\mathbb{R},\ H(p,v)=\langle p,v\rangle.$$
Fixing $v\in\mathbb{S}^{3}$, the singularities of the height function $h_{v}$ measures the contact of $M$ with the hyperplane orthogonal to $v$, denoted by $\Gamma_{v}$. This contact is also described by the one obtained using the fibers $\{g=0\}$ from Theorem \ref{classification}. Using a local parametrisation of $M$
given by Theorem \ref{teo.normal-form}, we have
$$h_{v}(x,y)=xv_{1}+xyv_{2}+\sum_{i+j=2,3}b_{ij}x^{i}y^{j}v_{3}+c_{20}x^{2}v_{4}+\sum_{i+j=3}c_{ij}x^{i}y^{j}v_{4},$$
for $v=(v_{1},v_{2},v_{3},v_{4})\in\mathbb{S}^{3}$.

The height function $h_{v}$ is singular at the origin if and only if $v_{1}=0$. Geometrically, this means that $\Gamma_{v}$ contains $T_{p}M$. Hence, if $v_{1}\neq0$, $h_{v}$ is regular and the fiber $\Gamma_{v}$ is transversal to $C_{p}M$ and contains $E_{p}$. This contact is also described by the contact of the zero fiber of $g_{1}=X$ with the model surface $\texttt{X}$.

Consider $S\subset\mathbb{R}^{4}$ the associated regular surface of $M$, as done before (see Theorem \ref{teorelation}). Given a binormal direction of $M$, $\nu\in N_{p}M$, $\textbf{u}$ will denote the corresponding asymptotic direction (which is also an asymptotic direction of $S$). Furthermore, $\tau$ is the torsion of the normal section of the surface $S$ tangent to the asymptotic direction $\textbf{u}$. Let $CS(\varepsilon)$ be the canal hypersurface of $S$. We denote by $\mathcal{C}$ the curve of cuspidal edge points of its Gauss map

\begin{prop}
Let $v=(0,v_{2},v_{3},0)$ with $v_{3}\neq0$. the hyperplane $\Gamma_{v}$ is tangent to $T_{p}M$ and transversal to $C_{p}M$ and $E_{p}$. The height function $h_{v}$ can have singularities of type $A_{k-1}^{\pm}$, $k=2,3,4$ which are modeled by the contact of the
zero fibre of the submersions $g_{2k}=\pm Z+(-1)^{k+1}X^{k}$ with the model surface $\texttt{X}$ (i.e. modeled by the composition of the submersions with the parametrisation of the model surface), respectively. It has a singularity of type $A_{1}$ (Morse) if and only if $v\in N_{p}M$ is not a binormal direction. For more degenerate singularities, this configuration has three possibilities:
\begin{itemize}
    \item[(i)] If $p$ is a hyperbolic point, the singularity is of type $A_{2}$ iff $v$ is a binormal direction of $M$ and $\tau\neq0$. Finally, the height function has an $A_{3}$ singularity iff $v$ is a binormal direction, $\tau=0$ and the asymptotic direction $\textbf{u}$ of $S$ is transversal to the curve $\mathcal{C}$ of cuspidal edge points of the Gauss map. See Table \ref{sing2}.
    \item[(ii)] If $p$ is a parabolic point, $h_{v}$ has singularity of type $A_{2}$ iff $v$ is a binormal direction of $M$ and the associated asymptotic direction $\textbf{u}$ is transversal to the parabolic curve $\delta$ of $S$. The singularity is of type $A_{3}$ iff $v$ is a binormal direciton and $\textbf{u}$ is tangent to $\delta$ with first order contact.
    \item[(iii)] If $p$ is elliptic, the height function can only have singularity of type $A_{1}$.
\end{itemize}
\end{prop}
\begin{proof}
The proof follows from Theorem \ref{teo-h-hiper} and Theorem \ref{teorelation} since both surfaces $M$ and $S$ have the same height function. However we will, present some calculations for the case $p$ hyperbolic, that is, $b_{20}>0$.
Let  $v=(0,v_{2},v_{3},0)$ with $v_{3}\neq0$. For the normal form in Theorem \ref{teo.normal-form}, $E_{p}$ is the $YZ$-plane and the tangent cone $C_{p}M$ is the $XZ$-plane. Hence, $\Gamma_{v}$ is transversal to $E_{p}$ and $C_{p}M$. So this situation is modeled by the zero fiber of $g=\pm Z+(-1)^{k+1}X^{k}$, $k=2,3,4$ and the model surface $\texttt{X}$.
Taking $v=(0,v_{2},1,0)$, the height function is given by
$$h_{v}(x,y)=(b_{11}+v_{2})xy+b_{20}x^{2}+b_{02}y^{2}+b_{30}x^{3}+b_{21}x^{2}y+b_{12}xy^{2}+b_{03}y^{3},$$
where $b_{02}>0$.
The determinant of the Hessian matrix of $h_{v}$ is given by $\det(\mathcal{H}(h_{v}(x,y)))=4b_{20}b_{02}-(v_{2}+b_{11})^{2}$. So, $h_{v}$ has a singularity of type $A_{1}$ (Morse) if and only if, $v_{2}\neq-b_{11}\pm2\sqrt{b_{20}b_{02}}$, which is equivalent to $v$ not being a binormal direction (see \cite{Benedini/Sinha/Ruas}). The conditions for $h_{v}$ to have a singularity of type $A_{2}$ are: $v$ is a binormal direction and
$$b_{30}\mp \frac{b_{21}\sqrt{b_{20}b_{02}}}{b_{02}}+\frac{b_{12}b_{20}}{b_{02}}\mp\frac{b_{03}b_{20}\sqrt{b_{20}b_{02}}}{b_{02}^{2}}\neq0,$$
and this last condition is exactly $\tau\neq0$, where $\tau$ is the torsion of the normal section along the asymptotic direction $\textbf{u}=(u_{1},\mp \sqrt{b_{20}b_{02}}u_{1}/b_{02})$, $u_{1},\neq0$, associated to $v$.
\end{proof}

The singularities of the height function $h_{v}$ at a hyperbolic point are presented in Table \ref{sing2}. For each possibility of $v\in\mathbb{S}^{3}$ we give the relative position of $\Gamma_{v}$, $E_{p}$ and $C_{p}M$, in addition to the submersion whose contact of the zero fibre with the model surface $\texttt{X}$ models the singularity type.


\begin{table}[h]
\caption{Types of singularities of $h_{v}$ (hyperbolic point)}
\centering{
\begin{tabular}{ccc}
\hline
Vector & Singularity type & submersion \\
\hline
$v=(1,0,0,0)$  &  submersion  & $g_{1}=X$ \\
 $E_{p}\subset\Gamma_{v}\pitchfork T_{p}M,C_{p}M$  & & \\
&  & \\
$v=(0,v_{2},v_{3},0)$ & $A_{1}\Leftrightarrow v\ \mbox{is not binormal}$  &  \\
 $\Gamma_{v}\pitchfork E_{p},C_{p}M$& $A_{2}\Leftrightarrow v\ \mbox{is binormal and}\ \tau\neq0$  & $g_{2k}^*=\pm Z+(-1)^{k+1}X^{k}$ \\ & $A_{3}\Leftrightarrow v\ \mbox{is binormal},\ \tau=0\ \mbox{and}\ \textbf{u}\pitchfork\mathcal{C}$.  & \\
& & \\
$v=(0,v_{2},0,0)$ & $A_{1}$ &  $g_{3}=Y$\\
$C_{p}M\subset\Gamma_{v}\pitchfork E_{p}$ &  & \\
&  & \\
$v=(0,0,0,v_{4})$  & $A_{2}\Leftrightarrow \kappa_{u}(p)\neq0$  & $g_{4}=\pm W\pm X^{2}$ \\
$E_{p},C_{p}M\subset\Gamma_{v},$ &  & \\
\hline
\end{tabular}
}
*$k=2,3,4$
\label{sing2}
\end{table}

\begin{coro}
The hyperplane $\Gamma_{v}$ is an osculating hyperplane if and only if it is transversal to $E_{p}$ and the height function has singularity of type $A_{\geqslant2}$.
\end{coro}

\begin{prop}
Let $v=(0,v_{2},0,0)$, $v_{2}\neq0$, the hyperplane $\Gamma_{v}$ contains the tangent cone $C_{p}M$ and is transversal to $E_{p}$. The height function has singularity of type $A_{1}$, which is described by the contact of the zero fiber of the submersion $g_{3}=Y$ with the model surface $\texttt{X}$.
\end{prop}
\begin{proof}
When $v=(0,v_{2},0,0)$, $v_{2}\neq0$, we can take $v=(0,1,0,0)$ and the height function is given by $h_{v}(x,y)=xy+o(4)$, whose singularity is of type $A_{1}$.
\end{proof}

\begin{prop}
Let $v=(0,0,0,v_{4})$, $v_{4}\neq0$. The hyperplane $\Gamma_{v}$ contains both $E_{p}$ and $C_{p}M$. The height function $h_{v}$ has singularity of type $A_{2}$, which is described by the contact of the zero fiber of the submersion $g_{4}=\pm W\pm X^{2}$ with the model surface $\texttt{X}$ if and only if $\kappa_{u}(p)\neq0$.
\end{prop}
\begin{proof}
Taking $v=(0,0,0,1)$, the height function is given by $h_{v}(x,y)=c_{20}x^{2}+\sum_{i+j=3}c_{ij}x^{i}y^{j}$. It has singularity of type $A_{2}$ if and only if $c_{20}\neq0$, which is equivalent to $\kappa_{u}(p)=2|c_{20}|\neq0$.
\end{proof}

\end{document}